\documentclass[journal]{IEEEtran}
\usepackage{xcolor}
\usepackage{amssymb}
\usepackage{amsmath}
\usepackage{amsthm}
\usepackage{lineno}
\usepackage{graphicx}
\usepackage{hyperref}

\usepackage{cite}
\usepackage{algorithm}
\usepackage{algorithmicx}
\usepackage{algpseudocode}
\usepackage{array}
\usepackage[caption=false,font=normalsize,labelfont=sf,textfont=sf]{subfig}
\usepackage{fixltx2e}
\usepackage{url}
\usepackage{booktabs}

\newtheorem{theorem}{Theorem}
\newtheorem{lemma}{Lemma}
\newtheorem{assumption}{Assumption}

\newtheorem{remark}{Remark}
\newcommand\bH{\mathbf{H}}

\begin{document}

%
\title{Variance-Reduced Stochastic Quasi-Newton Methods for Decentralized Learning: Part II}
\author{Jiaojiao Zhang{$^1$}, Huikang Liu{$^2$}, Anthony Man-Cho So{$^1$}, and Qing Ling{$^3$}
	\thanks{Jiaojiao Zhang and Anthony Man-Cho So are with the Department of Systems Engineering and Engineering Management, The Chinese University of Hong Kong.}%
	\thanks{Huikang Liu is with the Business School, Imperial College London.}
	\thanks{Qing Ling is with the School of Computer Science and Engineering and Guangdong Province Key Laboratory of Computational Science, Sun Yat-Sen University, as well as the Pazhou Lab.}%
}
\maketitle
\begin{abstract}
 In Part I of this work, we have proposed a general framework of decentralized stochastic quasi-Newton methods, which converge linearly to the optimal solution under the assumption that the local Hessian inverse approximations have  bounded positive eigenvalues. In Part II, we specify two fully decentralized stochastic quasi-Newton methods, damped regularized limited-memory DFP (Davidon-Fletcher-Powell) and damped limited-memory BFGS (Broyden-Fletcher-Goldfarb-Shanno), to locally construct such Hessian inverse approximations without extra sampling or communication.
 Both of the methods use a fixed moving window of $M$ past local gradient approximations and local decision variables to adaptively construct positive definite Hessian inverse approximations with bounded eigenvalues, satisfying the assumption in Part I for the linear convergence. For the proposed damped regularized limited-memory DFP, a regularization term is added to improve the performance. For the proposed damped limited-memory BFGS, a two-loop recursion is applied, leading to low storage and computation complexity. Numerical experiments demonstrate that the proposed quasi-Newton methods are much faster than the existing decentralized stochastic first-order algorithms.
\end{abstract}

\begin{IEEEkeywords}
decentralized optimization, stochastic
quasi-Newton methods, damped limited-memory DFP, damped limited-memory BFGS
\end{IEEEkeywords}

%
\IEEEpeerreviewmaketitle

\section{Introduction}
\label{sec:intro}
With the explosive growth of big data and the urgent need for privacy protection, decentralized learning has become attractive. In decentralized learning, local machines store large-scale data and collaboratively train models. In Part I of this work, we have considered a decentralized learning problem over an undirected and connected network with $n$ nodes, in the form of
\begin{align}\label{c1}
{x}^* =\arg\min_{x\in\mathbb{R}^{d}} ~ F(x) \triangleq  \frac{1}{n} \sum_{i=1}^{n} f_{i}(x).
\end{align}
Here, $x$ is the decision variable and $f_{i}$ is the average of $m_i$ sample costs such that $$f_i(x)\triangleq \frac{1}{m_i} \sum_{l=1}^{m_i} f_{i,l}(x),$$ where $f_{i,l}: \mathbb{R}^d\to \mathbb{R}$ is the $l$-th sample cost on node $i$ and $f_{i,l}$ is not accessible by any other nodes. The network is described by an undirected and connected graph $\mathcal{G}=(\mathcal{V}, \mathcal{E})$ with node set $\mathcal{V}=\{1,\ldots,n\}$ and edge set $\mathcal{E}\subseteq \mathcal{V}\times \mathcal{V}$.  Nodes $i$ and $j$ are neighbors and allowed to communicate with each other if and only if they are connected with an edge $(i,j)\in \mathcal{E}$. We define $\mathcal{N}_i$ as the set of neighbors of node $i$ including itself. All the nodes cooperate to find the optimal solution $x^*$ to \eqref{c1} using computation on their local costs $f_i$ and information received from their neighbors.

In Part I, we have established a general framework to solve \eqref{c1}, incorporating quasi-Newton approximations with variance reduction so as to achieve fast convergence. With initializations $\tau_i^0=x_i^0$ and  $g_i^0=v_i^0=\nabla f_i(x_i^0)$, at time $k+1$ node $i$ updates its local decision variable $x_i^{k+1}$ as
\begin{align}\label{z2}
x_i^{k+1}=&\sum_{j=1}^{n} w_{ij}x_j^k-\alpha d^k_i,\\\nonumber
v_i^{k+1}=&\frac{1}{b_i}\sum_{l\in S^{k+1}_i}\Big(\nabla f_{i,l}(x_i^{k+1})-\nabla f_{i,l}(\tau_i^{k+1})\Big)+\nabla f_{i}(\tau_i^{k+1}), \\\nonumber
g_i^{k+1}=&\sum_{j=1}^{n} w_{ij}g_j^{k}+v_i^{k+1}-v_i^k,\\\nonumber
d_i^{k+1}=&H_i^{k+1} g_i^{k+1}.
\end{align}
Here, $\alpha>0$ is the step size, $W=[w_{ij}] \in \mathbb{R}^{n \times n}$ is the mixing matrix, $S_i^{k+1} \subseteq \{1,\ldots,m_i\}$ with batch size $b_i$, while $\tau_i^{k+1}=\tau_i^{k}$ or $\tau_i^{k+1}=x_i^{k+1}$ if $\hspace{-0.5em} \mod(k+1,T)=0$. The general framework \eqref{z2} can be written in a compact form, as
\begin{align}\label{eq:alg}
\begin{aligned}
\mathbf{x}^{k+1}&=\mathbf{W}\mathbf{x}^{k}-\alpha \mathbf{d}^{k},\\
\mathbf{g}^{k+1}&=\mathbf{W}\mathbf{g}^{k}+\mathbf{v}^{k+1}-\mathbf{v}^{k},\\
\mathbf{d}^{k+1}&=\bH^{k+1}\mathbf{g}^{k+1},\\
\end{aligned}
\end{align}
where the notations can be found in Part I. 

At each time $k$, each node $i$ computes a local approximate Newton direction $d_i^k$ by the local gradient approximation $g_i^k$ and the local Hessian inverse approximation $H_i^k$. The local gradient approximation $g_i^k$ is obtained by the dynamic average consensus method  \cite{zhu2010discrete} to track the average of the variance-reduced local stochastic gradient $v_i^k$, while the local Hessian inverse approximation $H_i^k$ is constructed by quasi-Newton methods with the local decision variable $x_i^k$ and the gradient approximation $g_i^k$.
Part I proves that the proposed general framework \eqref{eq:alg} converges linearly to the optimal solution of \eqref{c1}, given that the Hessian inverse approximations $H_i^{k}$  satisfy the following assumption.
\begin{assumption}\label{asm-Hbound}
	There exist two constants $M_1$ and $M_2$ with $0< M_1\leq M_2<\infty$ such that
	\begin{equation}\label{b1}
	M_1I_d\preceq H^{k}_i\preceq M_2I_d, \; \forall \; i=1,\ldots,n,  \; k\ge 0.
	\end{equation}
\end{assumption}
In Part II of this work, we focus on on how to construct Hessian inverse approximations $H_i^k$ satisfying Assumption \ref{asm-Hbound},  with fully decentralized quasi-Newton methods.

We first review deterministic and stochastic quasi-Newton methods for solving \eqref{c1} in the centralized setting and then move on to the decentralized setting. In the centralized deterministic setting,  quasi-Newton methods usually update in the form of
\begin{equation*}
x^{k+1}=x^k-\alpha H^{k} \nabla F(x^k),
\end{equation*}
where $H^k$ is an approximation to $(\nabla^2 F(x^k))^{-1}$ and $\alpha>0$ is the step size.
The two well-known quasi-Newton methods, DFP (Davidon-Fletcher-Powell) and
BFGS (Broyden-Fletcher-Goldfarb-Shanno), update $H^k$ via
\begin{align}\label{b6}
(DFP) \quad  H^{k+1}=H^k+\frac{s^k(s^k)^T}{(s^k)^T{y}^k}-\frac{H^k{y}^k({y}^k)^TH^k}{(H^k{y}^k)^T {y}^k},
\end{align}
and
\begin{align}\label{b7}
(BFGS) \quad
H^{k+1}=&H^{k}-\frac{H^{k}y^k({s}^k)^T+{s}^k(y^k)^T H^k }{({s}^k)^Ty^k}\\\nonumber
&+\frac{{s}^k({s}^k)^T}{({s}^k)^Ty^k}\left(1+\frac{(y^k)^TH^{k}y^k}{({s}^k)^Ty^k}\right),
\end{align}
respectively. Here, $s^k$ and $y^k$ are defined as
$$s^k=x^{k+1}-x^{k}, \quad y^k=\nabla F(x^{k+1})-\nabla F(x^k).$$
If the cost function $F(x)$ is strongly convex, the curvature condition $(s^k)^Ty^k>0$ holds and thus the Hessian inverse approximations $H^k$ via \eqref{b6} and \eqref{b7} preserve positive definiteness given a positive definite initialization such that $H^0 \succ 0$ \cite{nocedal2006numerical}.

When the number of samples is very large, computing the full gradient $\nabla F$ is prohibitive, which motivates the development of stochastic methods. In the centralized stochastic setting, there are many works which combine stochastic gradient descent with carefully constructed curvature information \cite{bordes2009sgd,byrd2011use,byrd2016stochastic,schraudolph2007stochastic }.  For example,  \cite{bordes2009sgd} investigates how to construct a diagonal or low-rank matrix according to the secant condition. The work of \cite{byrd2011use} incorporates sub-sampled Hessian information in a Newton conjugate gradient method and a limited-memory quasi-Newton method for statistical learning. An online limited-memory BFGS using stochastic gradients is proposed in \cite{schraudolph2007stochastic}, in lieu of the full gradient in BFGS update; the convergence analysis is given in \cite{mokhtari2015global}. The work of
\cite{mokhtari2014res} proposes a regularized stochastic BFGS (RES) method where stochastic gradients are used both as descent directions and constituents of Hessian estimates. The regularization technique ensures that the eigenvalues of the Hessian approximations are uniformly bounded. The works of \cite{lucchi2015variance} and \cite{moritz2016linearly} take advantage of variance reduction to eliminate the stochastic gradient noise, such that the resultant stochastic quasi-Newton methods are provably convergent at linear rates.

Although the deterministic and stochastic
quasi-Newton methods have been widely used in the centralized setting, they cannot be used directly in decentralized optimization. Taking the decentralized network topology into consideration, each node is only allowed to communicate with its neighbors, which leads to the lack of global gradient and Hessian information. In the decentralized deterministic setting, there are a few works exploring the decentralized quasi-Newton methods with the penalization technique \cite{eisen2017decentralized, jerinkic2019distributed}, in the dual domain \cite{eisen2016decentralized}, and in the primal-dual domain \cite{eisen2019primal}. However, to the best of our knowledge, computationally affordable decentralized stochastic second-order methods have not been investigated.

In Part II of this work, we propose two fully decentralized quasi-Newton methods to construct local Hessian inverse approximations $H_i^k$ fitting into the general framework: All $H_i^k$ satisfy Assumption \ref{asm-Hbound}, and are constructed only with the local decision variables $x_i^k$ and the local gradient approximations $g_i^k$. Note that using the gradient approximations to construct the Hessian inverse approximations is quite adventurous, since the gradient approximations are noisy due to stochastic
gradient noise and disagreement among the nodes. Naively adopting centralized quasi-Newton methods may end up with almost-singular Hessian inverse approximations, or even non-positive semidefinite ones. To tackle these issues, the proposed methods use the damping and limited-memory techniques so as to adaptively construct positive definite Hessian inverse approximations with bounded eigenvalues.

\textbf{Notations.}
We use $\|\cdot\|$ to denote the Euclidean norm of a vector. $\text{tr}(\cdot)$, $\|\cdot\|_F$ and $\|\cdot\|_2$ denote the trace, the Frobenius norm and the spectral norm of a matrix, respectively. $I_d\in \mathbb{R}^{d\times d}$ denotes the $d\times d$ identity, and $1_n\in \mathbb{R}^{n}$ denotes the $n$-dimensional column vector of all ones. $A\succeq 0$ and $A\succ 0$ mean that $A$ is positive semidefinite  and  positive definite, respectively.  We use  $A\succeq B$ and $A\succ B$  to denote $A-B \succeq 0$ and $A-B\succ 0$, respectively.  $\lambda_{\max}(\cdot)$, $\lambda_{\min}(\cdot)$ denote the largest and smallest eigenvalues of a matrix, respectively. The $i$-th largest eigenvalue of a matrix is denoted by $\lambda_i(\cdot)$. We use $A^{\frac{1}{2}}$ to denote the square root of a positive semidefinite matrix $A$ such that $A=A^{\frac{1}{2}}A^{\frac{1}{2}}$. Define the aggregated variable  $\mathbf{x}=[x_1;\cdots;x_n]\in \mathbb{R}^{nd}$ for $x_1, \cdots, x_n \in \mathbb{R}^d$, and similar aggregation rules apply to other variables $ \mathbf{d}, \mathbf{g}$, and $\mathbf{v}$.
Define a block diagonal matrix $\bH^{k}=\operatorname{diag}\{H_i^{k}\}\in \mathbb{R}^{nd\times nd}$ whose $i$-th block is $H_i^{k} \in \mathbb{R}^{d\times d}$.

\section{Damped Regularized Limited-memory DFP for Hessian Inverse Approximation} \label{sec-dfp}
In this section, we propose a fully decentralized stochastic quasi-Newton approximation approach called the damped regularized limited-memory DFP, to locally construct the Hessian inverse approximations. Each node $i$ only uses its own gradient approximations $g_i^{k+1}$ and decision variables $x_i^{k+1}$ to construct $H_i^{k+1}$, and does not need extra communication with its neighbors. Therefore, from now on we do not specify the node index $i$ until the end of derivation. 

As we have emphasized in Section \ref{sec:intro}, constructing a reliable Hessian inverse approximation is challenging in the decentralized stochastic setting. Since the gradient approximations $g^{k+1}$ are noisy due to stochastic gradient noise and disagreement among the nodes, naively adopting centralized quasi-Newton methods may lead to almost-singular Hessian inverse approximations (i.e., either  $\lambda_{\min}(H^{k+1})\to 0$ or $\lambda_{\max}(H^{k+1})\to \infty$) or even non-positive semidefinite ones.


\subsection{Damped Regularized Limited-memory DFP}

%
It is known that the update of DFP is obtained by minimizing the Gaussian differential entropy subject to certain constraints. Inspired by \cite{mokhtari2014res,chen2019stochastic}, to avoid $\lambda_{\min}(H^{k+1})\to 0$,  we add a regularization term with parameter $\rho>0$ to the minimization problem, given by
\begin{align}\label{3}
H^{k+1}=\arg\min_{Z\in \mathbb{R}^{d \times d}}& \;\text{tr}[(H^k)^{-1}(Z-\rho I_d)] \\ \nonumber
&-\log\det[(H^k)^{-1}(Z-\rho I_d)], \\ \nonumber
\text{s.t.}&\; Z{y}^k=s^k,~ Z\succeq 0,
\end{align}
where $s^k =  x^{k+1} - x^k$ is variable variation and $y^k = g^{k+1} - g^k$ is gradient approximation variation. If we let $\rho=0$, then \eqref{3} reduces to the traditional DFP.

Define a modified variable variation $\hat{s}^k$ as
\begin{align}\label{4}
\hat{s}^k=s^k-\rho y^k,
\end{align}
The work \cite{mokhtari2014res} has proved that the closed-form solution to \eqref{3} is given by
\begin{align}\label{5}
H^{k+1}=H^k+\frac{\hat{s}^k(\hat{s}^k)^T}{(\hat{s}^k)^Ty^k}-\frac{H^ky^k(y^k)^TH^k}{(y^k)^T H^k y^k}+\rho I_d.
\end{align}
In the centralized stochastic setting where $g^k$ is the stochastic gradient instead of the stochastic gradient approximation, \cite{mokhtari2014res} shows that if $\rho$ is properly chosen and $H^0 \succ 0$, then $(\hat{s}^k)^Ty^k>0$ and thus $\lambda_{\min}(H^{k+1})>\rho$  for all $k$. In addition, \cite{mokhtari2014res} also establishes the upper bound on the eigenvalues of $H^{k+1}$ for all $k$. Unfortunately, these satisfactory results no longer hold in the decentralized stochastic setting, since $y^k = g^{k+1} - g^k$ suffers from both stochastic gradient noise and disagreement among the nodes. To address this issue, we further propose the damped regularized limited-memory DFP as follows.

To preserve positive semidefiniteness and lower boundedness of $H^k$, we combine a damping technique with the regularized DFP in \eqref{5}. To be specific, given a suitable $H^0\succ 0$, we define $\hat{y}^k$ as
\begin{equation}\label{6}
\hat{y}^k=\theta^k y^k+(1-\theta^k)(H^0+\epsilon I_d)^{-1} \hat{s}^k,
\end{equation}
where $\epsilon>0$ is a constant and $\theta^k$ is adaptively computed by
\begin{align}\label{ineq:thetak}
	\theta^k=\min\left\{\tilde{\theta}^k,\frac{ \tilde{L} \|\hat{s}^k\|}{\|y^k\|}\right\},
\end{align}
with a parameter $\tilde{L}>0$ and
the widely used parameter $\tilde{\theta}^k$ defined as \cite{chen2019stochastic}
\begin{align}\label{7}
\hspace{-1.1em}{\tilde{\theta}}^k=\left\{\begin{array}{l}
\frac{0.75 (\hat{s}^k)^{T}\left(H^{0}+\epsilon I_d\right)^{-1} \hat{s}^k}{(\hat{s}^k)^{T}\left(H^{0}+\epsilon I_d\right)^{-1} \hat{s}^k-(\hat{s}^k)^{T}y^k}, \\
\hspace{1em} \text { if } (\hat{s}^k)^{T}y^k \leq 0.25 (\hat{s}^k)^{T}\left(H^{0}+\epsilon I_d\right)^{-1} \hat{s}^k,\\
1, \text { otherwise.}
\end{array}\right.
\end{align}
As we will show later, with the added term $\frac{\tilde{L}\|\hat{s}^k\|}{\|y^k\|}$ in \eqref{ineq:thetak},  $\|\hat{y}^k\|$ can be upper bounded in terms of $\|\hat{s}^k\|$. 
Then, we replace $y^k$ with $\hat{y}^k$ in \eqref{5} to obtain the damped regularized DFP update
\begin{align}\label{a11}
H^{k+1}=H^k+\frac{\hat{s}^k(\hat{s}^k)^T}{(\hat{s}^k)^T\hat{y}^k}-\frac{H^{k}\hat{y}^k(\hat{y}^k)^TH^{k}}{(\hat{y}^k)^TH^{k}\hat{y}^k}+\rho I_d.
\end{align}	
As we will prove in Lemma \ref{lem-dfp-pd}, the damping technique guarantees that $(\hat{s}^k)^T\hat{y}^k>0$, such that $\lambda_{\min}(H^{k+1})>\rho$.

On the other hand, we also have to guarantee $\lambda_{\max}(H^{k+1})$ $< \infty$. This is nontrivial since $\hat{y}$ is noisy and the regularization term $\rho I_d$ may accumulate when the algorithm evolves.  Inspired by \cite{wang2017stochastic,chen2019stochastic}, we use the limited-memory technique to tackle this issue. Usually, the limited-memory technique is combined with BFGS to reduce the memory and computation costs. In contrast, here we combine it with DFP to make the eigenvalues of $H^k$ bounded.

From now on, we specify the node index $i$. For time $k$, set an initial Hessian inverse approximation
\begin{align}\label{a12}
H_i^{k,(0)}=\min\left\{\max\left\{\frac{(s_i^k)^Ts_i^k}{(s_i^k)^Ty_i^k}+{\rho},{\beta}\right\},\mathcal{B}\right\}I_d,
\end{align}
where $\beta>0$ and $\mathcal{B}>0$ are two parameters to guarantee that $\beta I_d \preceq H_i^{k,(0)}\preceq\mathcal{B}I_d$ in initialization.
Given the two sequences $\{\hat{s}_i^p\}$ and $\{\hat{y}_i^p\}$, $p=k+1-\tilde{M}, \ldots, k$, where $\tilde{M}=\min\{k+1,M\}$ and $M$ is the memory size, the damped regularized limited-memory DFP at node $i$ updates as
\begin{align}\label{a13}
H_i^{k,(t+1)}=&H_i^{k,(t)}+\frac{\hat{s}_i^p(\hat{s}_i^p)^T}{(\hat{s}_i^p)^T\hat{y}_i^p} \\
              &-\frac{H_i^{k,(t)}\hat{y}_i^p(\hat{y}_i^p)^TH_i^{k,(t)}}{(\hat{y}_i^p)^TH_i^{k,(t)}\hat{y}_i^p}+\rho I_d, \notag
\end{align}
where $p=k+1-\tilde{M}+t$ and $t=0,\ldots,\tilde{M}-1$. At the end of this inner loop, we set $H_i^{k+1}=H_i^{k,(\tilde{M})}$.

The damped regularized limited-memory DFP for Hessian inverse approximation at node $i$ for time $k$ is summarized in Algorithm \ref{alg-DFP}.
\floatname{algorithm}{Algorithm}
\begin{algorithm}[htbp]
	\caption{Damped regularized limited-memory DFP for Hessian inverse approximation at node $i$ for time $k$} \label{alg-DFP}
	\begin{algorithmic}[1]
		\Require  $\rho$;  $\beta$; $\mathcal{B}$; $\epsilon$; $\tilde{L}$; $M$.
		\State Update variable variation
		$s_i^k=x_i^{k+1}-x_i^{k}.$
		\State Update gradient variation
		$y_i^k=g_i^{k+1}-g_i^{k}.$
		\State Update modified variable variation
		$\hat{s}_i^k=s_i^k-\rho y_i^k.$	
		\State Update modified gradient variation $\hat{y}_i^k$ as in \eqref{6}.
		\State Set $\tilde{M}=\min\{k+1,M\}$ and load $\{\hat{s}_i^p,\hat{y}_i^p\}_{p=k+1-\tilde{M}}^{k}$.
		\State Initialize $H_i^{k,(0)}$ as in \eqref{a12}.	
		\For {{$t=0,\ldots,\tilde{M}-1$}}
		\State Update $H_i^{k,(t+1)}$ as in \eqref{a13}.
		\EndFor
		\State Output $H_i^{k+1}=H_i^{k,(\tilde{M})}$.
	\end{algorithmic}
\end{algorithm}

\subsection{Bounded Positive Eigenvalues of $H_i^k$ Constructed by DFP}
In this section, we will prove that the Hessian inverse approximations  constructed  by the proposed damped regularized limited-memory DFP are positive definite and have bounded eigenvalues.

Consider the update \eqref{a13}. The following Lemma shows the damping technique guarantees that $(\hat{s}_i^p)^T\hat{y}_i^p>0$, which yields $\lambda_{\min}(H_i^{k})>\rho$.
\begin{lemma}\label{lem-dfp-pd}
Considering the update in \eqref{a13}, with the corrected $\hat{y}^p$ by the damping technique, we have
\begin{equation*}
\text{$0<\theta_i^p\le 1$ and
$(\hat{s}_i^p)^T\hat{y}_i^p\ge 0.25 (\hat{s}_i^p)^T(H_i^{k,(0)}+\epsilon I)^{-1}\hat{s}_i^p$.}
\end{equation*}	
Moreover, with the initialization $H_i^{k,(0)}$ defined in \eqref{a12},  $H_i^{k+1}$ generated by the damped regularized limited-memory DFP in Algorithm \ref{alg-DFP} on node $i$ keeps positive definite, such that $\lambda_{\min}(H_i^{k+1})>\rho$.
\end{lemma}

\begin{proof}
	Since the analysis holds for any node, we omit the node index $i$ in the proof.
	With $\tilde{\theta}^p$ defined in \eqref{7}, if $(\hat{s}^p)^{T}y^p > 0.25 (\hat{s}^p)^{T}\left(H^{k,(0)}+\epsilon I_d\right)^{-1} \hat{s}^p$, we have $\tilde{\theta}^p=1$. If $(\hat{s}^p)^{T}y^p \le 0.25 (\hat{s}^p)^{T}\left(H^{k,(0)}+\epsilon I_d\right)^{-1} \hat{s}^p$, substituting this inequality into the definition of $\tilde{\theta}^p$, we have
	\begin{align*}
	&\tilde{\theta}^p=\frac{0.75 (\hat{s}^p)^{T}\left(H^{k,(0)}+\epsilon I_d\right)^{-1} \hat{s}^p}{(\hat{s}^p)^{T}\left(H^{k,(0)}+\epsilon I_d\right)^{-1} \hat{s}^p-(\hat{s}^p)^{T}y^p}\\
	\le& \frac{0.75 (\hat{s}^p)^{T}\left(H^{k,(0)}+\epsilon I_d\right)^{-1} \hat{s}^p}{(\hat{s}^p)^{T}\left(H^{k,(0)}+\epsilon I_d\right)^{-1} \hat{s}^p-[0.25 (\hat{s}^p)^{T}\left(H^{k,(0)}+\epsilon I_d\right)^{-1} \hat{s}^p]}\\
	=&1.
	\end{align*}
Obviously, with $H^{k,(0)}\succ 0$, we have $\tilde{\theta}^p> 0$. Since $
	\theta^p=\min\left\{\tilde{\theta}^p,\frac{\tilde{L}\|\hat{s}^p\|}{\|y^p\|}\right\}$, with $0<\tilde{\theta}^p\le 1$, we have $0<\theta^p\le 1$.

Moreover, substituting the definitions of $\hat{y}^p$ in \eqref{6} and $\theta^p$ in \eqref{ineq:thetak}, we compute
	\begin{align}
	&(\hat{s}^p)^T\hat{y}^p\\\nonumber
	=&(\hat{s}^p)^T\left(\theta^p y^p+(1-\theta^p)(H^{k,(0)}+\epsilon I)^{-1} \hat{s}^p\right)\\\nonumber
	=& \theta^p\left[(\hat{s}^p)^Ty^p-(\hat{s}^p)^T(H^{k,(0)}+\epsilon I)^{-1}\hat{s}^p\right]\\\nonumber
	&+(\hat{s}^p)^T(H^{k,(0)}+\epsilon I)^{-1}\hat{s}^p\\\nonumber
=&\left\{\begin{array}{l}
{0.25 (\hat{s}^p)^{T}\left(H^{k,(0)}+\epsilon I_d\right)^{-1} \hat{s}^p}, \\
\hspace{1em} \text { if } (\hat{s}^p)^{T}y^p \leq 0.25 (\hat{s}^p)^{T}\left(H^{k,(0)}+\epsilon I_d\right)^{-1} \hat{s}^p,\\
(\hat{s}^p)^T y^p, \text { otherwise,}
\end{array}\right.
\end{align}
which implies
\begin{equation*}
(\hat{s}^p)^T\hat{y}^p\ge 0.25 (\hat{s}^p)^{T}\left(H^{k,(0)}+\epsilon I_d\right)^{-1} \hat{s}^p.
\end{equation*}
By the initialization $H^{k,(0)}$ in \eqref{a12}, we have $(\hat{s}^p)^T\hat{y}^p>0$.

To show that $H^{k+1}$ is positive definite, we rewrite \eqref{a13} as
\begin{align*}
H^{k,(t+1)}=\frac{\hat{s} \hat{s}^{T}}{\hat{s}^{T} \hat{y}}+H^{\frac{1}{2}}\left(I_d-\frac{H^{\frac{1}{2}} \hat{y} (H^{\frac{1}{2}}\hat{y})^T}{(H^{\frac{1}{2}}\hat{y})^T H^{\frac{1}{2}} \hat{y}}\right) H^{\frac{1}{2}}+\rho I_d,
\end{align*}
where  we omit all the time indexes at the right-hand side of \eqref{a13} for simplicity. Since    $0\preceq\frac{H^{\frac{1}{2}} \hat{y} (H^{\frac{1}{2}}\hat{y})^T}{(H^{\frac{1}{2}}\hat{y})^T H^{\frac{1}{2}} \hat{y}}\preceq I_d$, with $\hat{s}^T\hat{y}>0$, we conclude that $H^{k,(t+1)}$ is positive definite and $\lambda_{\min}(H^{k,(t+1)})>\rho$, which completes the proof.
\end{proof}

Based on Lemma \ref{lem-dfp-pd}, the following theorem further gives the specific lower and upper bounds for the eigenvalues of $H_i^k$ generated by Algorithm \ref{alg-DFP}.
\begin{theorem}\label{thm-dfp}
Consider the damped regularized limited-memory DFP in Algorithm \ref{alg-DFP}. We have
\begin{equation*}
M_1 I_d \preceq H_i^k \preceq M_2 I_d,\forall i,
\end{equation*}
where $M_1=\rho+(1+\omega)^{-2M}\left(\frac{1}{\beta}+\frac{1}{4(\mathcal{B}+\epsilon)}\right)^{-1}$,  $M_2=\mathcal{B}+M(4\mathcal{B}+4\epsilon+\rho) $ and  $\omega=4(\mathcal{B}+\epsilon)\left(\tilde{L}+\frac{1}{\beta+\epsilon}\right)$.
\end{theorem}
\begin{proof}
Since the analysis holds for any node, we omit the node index $i$ in the proof.
First, we establish the upper bound.
According to \eqref{a13}, we know that $H^{k,(t+1)}\leq H^{k,(t)}+\frac{\hat{s}^p(\hat{s}^p)^T}{(\hat{s}^p)^T\hat{y}^p}+\rho I_d$, which implies
\begin{equation}\label{ineq:hkt1}
\begin{split}
	\|H^{k,(t+1)}\|_2&\leq \|H^{k,(t)}\|_2+\left\|\frac{\hat{s}^p(\hat{s}^p)^T}{(\hat{s}^p)^T\hat{y}^p}\right\|_2+\rho\\
	&\le \|H^{k,(t)}\|_2+\frac{\|\hat{s}^p\|^2}{(\hat{s}^p)^T\hat{y}^p}+\rho,
\end{split}
\end{equation}
where we use $(\hat{s}^p)^T\hat{y}^p>0$ in the last inequality.
Then, with Lemma \ref{lem-dfp-pd}, we have
\begin{align}\label{ineq:syp}
	(\hat{s}^p)^T\hat{y}^p\geq& 0.25 (\hat{s}^p)^{T}\left(H^{k,(0)}+\epsilon I_d\right)^{-1} \hat{s}^p\\
	\geq& \frac{0.25\|\hat{s}^p\|^2}{\mathcal{B}+\epsilon}\nonumber,
\end{align}
where the last inequality holds since $H^{0}\preceq\mathcal{B}I_d$. Substituting \eqref{ineq:syp} into \eqref{ineq:hkt1}, we get
\begin{equation}\label{ineq:hk-recur}
	\|H^{k,(t+1)}\|_2\leq \|H^{k,(t)}\|_2+4(\mathcal{B}+\epsilon)+\rho.
\end{equation}
 Following the standard argument for recurrence, from \eqref{ineq:hk-recur} we get
\begin{equation*}
	\begin{split}
		\|H^{k+1}\|_2=\|H^{k,(\tilde{M})}\|_2&\leq \|H^{k,(0)}\|_2+\tilde{M}(4\mathcal{B}+4\epsilon+\rho)\\
		&\leq \mathcal{B}+M(4\mathcal{B}+4\epsilon+\rho).
	\end{split}
\end{equation*}
Thus, we get the upper bound $M_2=\mathcal{B}+M(4\mathcal{B}+4\epsilon+\rho)$.

Next, we establish the lower bound. According to \eqref{6} and \eqref{ineq:thetak}, with the added term $\frac{\tilde{L}\|\hat{s}^p\|}{\|y^p\|}$ in \eqref{ineq:thetak}, $\|\hat{y}^p\|$ can be upper bounded in terms of $\|\hat{s}^p\|$ such that \begin{equation}\label{ineq:hatyp}
\begin{split}
\|\hat{y}^p\|&\leq\theta^p\|y^p\|+(1-\theta^p)\|(H^{k,(0)}+\epsilon I_d)^{-1} \hat{s}^p\|\\
&\leq \tilde{L}\|\hat{s}^p\|+\frac{1}{\beta+\epsilon}\|\hat{s}^p\|,
\end{split}
\end{equation}
where we use $0<\theta^p\le \frac{\tilde{L}\|\hat{s}^p\|}{\|y^p\|}$ in the second inequality.
Using Sherman-Morrison-Woodbury formula \cite{nocedal2006numerical} on \eqref{a13}, we get
\begin{align}\label{w20}
	&\left(H^{k,(t+1)}-\rho I_d\right)^{-1}\\\nonumber
=&\left(I_d-\frac{\hat{y}^p(\hat{s}^p)^T}{(\hat{s}^p)^T\hat{y}^p}\right)\left(H^{k,(t)}\right)^{-1}\left(I_d-\frac{\hat{s}^p(\hat{y}^p)^T}{(\hat{s}^p)^T\hat{y}^p}\right)+\frac{\hat{y}^p(\hat{y}^p)^T}{(\hat{s}^p)^T\hat{y}^p}.
\end{align}

With \eqref{ineq:syp} and \eqref{ineq:hatyp}, we bound the two terms at the right-hand side of \eqref{w20} as follows. The first term at the right-hand side of \eqref{w20}  satisfies
\begin{equation}\label{ineq:ysp-norm}
	\left\|\frac{\hat{y}^p(\hat{s}^p)^T}{(\hat{s}^p)^T\hat{y}^p}\right\|_2\leq \frac{\|\hat{y}^p\|\cdot\|\hat{s}^p\|}{(\hat{s}^p)^T\hat{y}^p}\leq 4(\mathcal{B}+\epsilon)\left(\tilde{L}+\frac{1}{\beta+\epsilon}\right),
\end{equation}
where we substitute \eqref{ineq:syp} and \eqref{ineq:hatyp} in the last inequality. For the second term at the right-hand side of \eqref{w20}, we have
\begin{equation}\label{ineq:y2}
\left\|\frac{\hat{y}^p(\hat{y}^p)^T}{(\hat{s}^p)^T\hat{y}^p}\right\|_2\leq \frac{\|\hat{y}^p\|^2}{(\hat{s}^p)^T\hat{y}^p}\leq 4(\mathcal{B}+\epsilon)\left(\tilde{L}+\frac{1}{\beta+\epsilon}\right)^2,
\end{equation}
where we also substitute \eqref{ineq:syp} and \eqref{ineq:hatyp} in the last inequality. For simplicity, let $\omega\triangleq 4(\mathcal{B}+\epsilon)\left(\tilde{L}+\frac{1}{\beta+\epsilon}\right)$, then taking norm on both sides of \eqref{w20}, we have
\begin{align}\label{ineq:hkinv}
		&\left\|\left(H^{k,(t+1)}-\rho I_d\right)^{-1}\right\|_2\\\nonumber
		\le& \left(1+\left\|\frac{\hat{y}^p(\hat{s}^p)^T}{(\hat{s}^p)^T\hat{y}^p}\right\|_2\right)^2 \cdot \left\|\left(H^{k,(t)}\right)^{-1}\right\|_2+\left\|\frac{\hat{y}^p(\hat{y}^p)^T}{(\hat{s}^p)^T\hat{y}^p}\right\|_2\\\nonumber
		\leq&(1+\omega)^2\left\|\left(H^{k,(t)}\right)^{-1}\right\|_2+\frac{\omega^2}{4(\mathcal{B}+\epsilon)}\\\nonumber
		\leq&(1+\omega)^2\left\|\left(H^{k,(t)}-\rho I_d\right)^{-1}\right\|_2+\frac{\omega^2}{4(\mathcal{B}+\epsilon)},
\end{align}
where we substitute \eqref{ineq:ysp-norm} and \eqref{ineq:y2} in the second inequality, and
use the fact that $H^{k,(t)}\succ H^{k,(t)}-\rho I_d\succ0$ in the last inequality.
 A standard argument for recurrence on \eqref{ineq:hkinv} shows that
\begin{align}\label{w24}
	&\left\|\left(H^{k,(\tilde{M})}-\rho I_d\right)^{-1}\right\|_2\\\nonumber
	\leq&(1+\omega)^{2(M-1)}\left(\left\|\left(H^{k,(1)}-\rho I_d\right)^{-1}\right\|_2+\frac{\frac{\omega^2}{4(\mathcal{B}+\epsilon)}}{(1+\omega)^2-1}\right)\\\nonumber
	\leq&(1+\omega)^{2(M-1)}\left(\left\|\left(H^{k,(1)}-\rho I_d\right)^{-1}\right\|_2+\frac{1}{4(\mathcal{B}+\epsilon)}\right),
\end{align}
where to derive the second inequality we use $$\frac{\frac{\omega^2}{4(\mathcal{B}+\epsilon)}}{(1+\omega)^2-1}=\frac{\omega}{(\omega+2)4(\mathcal{B}+\epsilon)} <\frac{1}{4(\beta+\epsilon)}.$$

Besides, by setting $t=0$ in the second inequality of \eqref{ineq:hkinv}, we know that
\begin{equation*}
\left\|\left(H^{k,(1)}-\rho I_d\right)^{-1}\right\|_2\leq (1+\omega)^2\beta^{-1}+\frac{\omega^2}{4(\mathcal{B}+\epsilon)}.
\end{equation*}
Substituting the above inequality into \eqref{w24}, we get
\begin{equation}\label{w25}
\begin{split}
&\left\|\left(H_i^{k,(\tilde{M})}-\rho I_d\right)^{-1}\right\|_2\\
\leq&(1+\omega)^{2(M-1)}\left((1+\omega)^2\beta^{-1}+\frac{1+\omega^2}{4(\mathcal{B}+\epsilon)}\right)\\
\leq&(1+\omega)^{2M}\left(\beta^{-1}+\frac{1}{4(\mathcal{B}+\epsilon)}\right),
\end{split}
\end{equation}
where we use $1+\omega^2\le (1+\omega)^2$ in the last inequality.
By taking inverse on both sides of \eqref{w25}, we get
\begin{equation*}
	\lambda_{\min} \left(H^{k,(\tilde{M})}\right)\geq\rho+(1+\omega)^{-2M}\left(\frac{1}{\beta}+\frac{1}{4(\mathcal{B}+\epsilon)}\right)^{-1}.
\end{equation*}
Thus, we obtain the lower bound of $M_1=\rho+(1+\omega)^{-2M}$ $\left(\frac{1}{\beta}+\frac{1}{4(\mathcal{B}+\epsilon)}\right)^{-1}$, which completes the proof.
\end{proof}
\section{Damped Limited-memory BFGS for Hessian Inverse Approximation} \label{sec-bfgs}
In this section, we propose a damped limited-memory BFGS method to construct the local Hessian inverse approximations. Compared with the damped regularized limited-memory DFP method in Section \ref{sec-dfp}, the limited-memory technique here is used not only for bounding the  Hessian inverse approximations, but also for reducing storage and computation costs.  To be specific, the proposed method can be implemented by a two-loop recursion,  where  $H_i^k$ is not generated explicitly, and only its multiplication with vectors are computed.

\subsection{Damped Limited-memory BFGS}

Since the discussion holds for any node, from now on we omit the node index $i$ until the end of derivation.
The traditional BFGS is the solution to an optimization problem given by
\begin{equation}\label{10}
\begin{aligned}
H^{k+1}=\arg\min_Z& \; \|Z-H^k\|_Q, \\
\text{s.t.}&\; Z{y}^k={s}^k, Z=Z^T.
\end{aligned}
\end{equation}
where $\|\cdot\|_Q$ is a weighted Frobenius norm defined as $\|H\|_Q=\|Q^{\frac{1}{2}}HQ^{\frac{1}{2}}\|_F$ and $Q$ is an arbitrary positive semidefinite matrix satisfying the relation $Q{s}^k={y}^k$ \cite{nocedal2006numerical}. Note that in our decentralized setting, $s^k =  x^{k+1} - x^k$ is variable variation and $y^k = g^{k+1} - g^k$ is gradient approximation variation. The solution of
the semidefinite optimization in \eqref{10} is therefore the closest to $H^k$ in the sense of weighted norm
among all symmetric matrices that satisfy the secant condition.
The closed form solution of \eqref{10} is given by
\begin{align}\label{11}
H^{k+1}=&H^{k}-\frac{H^{k}y^k({s}^k)^T+{s}^k(y^k)^T H^k }{({s}^k)^Ty^k}\\\nonumber
&+\frac{{s}^k({s}^k)^T}{({s}^k)^Ty^k}\left(1+\frac{(y^k)^TH^{k}y^k}{({s}^k)^Ty^k}\right). 
\end{align}

As we have mentioned in Section \ref{sec-dfp},  the gradient approximations $g^{k}$ are noisy such that it is nontrivial to preserve positive semidefiniteness and lower boundedness of $H^k$ (i.e., $\lambda_{\min}(H^{k})$ and $\lambda_{\max}(H^{k})$ must be both positive and finite). Naively adopting the BFGS update in \eqref{11} is unable to achieve this goal. Therefore, we combine the damping technique with the traditional BFGS method in \eqref{11}. To be specific, given a suitable $H^0\succ 0$, we define $\hat{y}^k$ as
\begin{equation}\label{12}
\hat{y}^k=\theta y^k+(1-\theta)(H^0+\epsilon I)^{-1} s^k,
\end{equation}
where $\epsilon>0$ is a parameter and $\theta$ is adaptively computed by
\begin{align}\label{ineq:thetak1}
	\theta^k=\min\left\{\tilde{\theta}^k,\frac{\tilde{L}\|s^k\|}{\|y^k\|}\right\},
\end{align}
with a parameter $\tilde{L}>0$ and the widely used parameter $\tilde{\theta}^k$ defined as \cite{chen2019stochastic}
\begin{align}\label{13}
\tilde{\theta}^k=\left\{\begin{array}{l}
\frac{0.75 (s^k)^{T}\left(H^{0}+\epsilon I_d\right)^{-1} s^k}{(s^k)^{T}\left(H^{0}+\epsilon I_d\right)^{-1} s^k-(s^k)^{T}y^k}, \\
\hspace{0.5em}\text { if } (s^k)^{T}y^k \leq 0.25 (s^k)^{T}\left(H^{0}+\epsilon I_d\right)^{-1} s^k,\\
1, \text { otherwise.}
\end{array}\right.
\end{align}
Similar to the proposed DFP-based method, with the added term $\frac{\tilde{L}\|{s}^k\|}{\|y^k\|}$ in \eqref{ineq:thetak1}, $\|\hat{y}^k\|$ can be upper bounded in terms of $\|s^k\|$. 
Then, we replace $y^k$ in \eqref{11} with $\hat{y}^k$, such that the damped BFGS becomes
\begin{align}\label{14}
H^{k+1}=&H^{k}-\frac{H^{k}\hat{y}^k({s}^k)^T+{s}^k(\hat{y}^k)^T H^k }{({s}^k)^T\hat{y}^k}\\\nonumber
&+\frac{{s}^k({s}^k)^T}{({s}^k)^T\hat{y}^k}\left(1+\frac{(\hat{y}^k)^TH^{k}\hat{y}^k}{({s}^k)^T\hat{y}^k}\right).
\end{align}
As we will prove later, the damping technique guarantees that  $({s}^k)^T\hat{y}^k>0$, such that $\lambda_{\min}(H^{k+1})>0 $.

On the other hand, to guarantee that $\lambda_{\max}(H^{k+1}) < \infty$, we use the limited-memory technique. From now on, we specify the node index $i$. For time $k$, set an initial Hessian inverse approximation as
\begin{equation}\label{w21}
H_i^{k,(0)}=\min\left\{\max\left\{\frac{(s_i^k)^Ty_i^k}{(y_i^k)^Ty_i^k},{\beta}\right\},\mathcal{B}\right\} I_d.
\end{equation}
Given the two sequences $\{{s}_i^p\}$ and $\{\hat{y}_i^p\}$, $p=k+1-\tilde{M}, \ldots, k$, where $\tilde{M}=\min\{k+1,M\}$ and $M$ is the memory size, the damped limited-memory BFGS at node $i$ updates as
\begin{align}\label{15}
&H_i^{k,(t+1)}=H_i^{k,(t)}-\frac{H_i^{k,(t)}\hat{y}_i^p({s}_i^p)^T+{s}_i^p(\hat{y}_i^p)^T H_i^{k,(t)} }{({s}_i^p)^T\hat{y}_i^p}\\\nonumber
&\hspace{4.7em}+\frac{{s}_i^p({s}_i^p)^T}{({s}_i^p)^T\hat{y}_i^p}\left(1+\frac{(\hat{y}_i^p)^TH_i^{k,(t)}\hat{y}_i^p}{({s}_i^p)^T\hat{y}_i^p}\right)\\\nonumber
&=\left(I_d-\frac{s_i^p (\hat{y}_i^p)^T}{(s_i^p)^T \hat{y}_i^p}\right)H_i^{k,(t)} \left(I_d-\frac{\hat{y}_i^p (s_i^p)^T}{(s_i^p)^T \hat{y}_i^p}\right)+\frac{s_i^p(s_i^p)^T}{(s_i^p)^T \hat{y}_i^p},
\end{align}
where $p=k+1-\tilde{M}+t$ and $t=0,\ldots,\tilde{M}-1.$ The second equality will be used for the analysis.
At the end of this inner loop, we set $H_i^{k+1}=H_i^{k,(\tilde{M})}$.

Compared with the proposed DFP-based method in Algorithm \ref{alg-DFP}, one advantage of
the proposed BFGS-based method is that the update \eqref{15} can be realized by a two-loop recursion,  where  $H_i^{k,(t)}$ is not generated explicitly, and only its multiplications with vectors are computed. The damped limited-memory BFGS for Newton direction approximation at node $i$ is summarized in Algorithm  \ref{alg-BFGS}. The two-loop recursion at node $i$ is summarized in Algorithm \ref{alg-loop}.

\floatname{algorithm}{Algorithm}
\begin{algorithm}[htbp]
	\caption{Damped limited-memory BFGS for Newton direction approximation run on agent $i$} \label{alg-BFGS}
	\begin{algorithmic}[1]
		\Require $\beta$; $\mathcal{B}$; $\epsilon$; $\tilde{L}$; $M$.
		\State Update variable variation
		$s_i^k=x_i^{k+1}-x_i^{k}.$
		\State Update the gradient variation
		$y_i^k=g_i^{k+1}-g_i^{k}.$
		\State Update modified gradient  variation $\hat{y}_i^k$ as in \eqref{12}.
		\State Initialize  $H_i^{k,(0)}$ as in \eqref{w21}.
		\State Set $\tilde{M}=\min\{k+1,M\}$ and load $\{{s}_i^p,\hat{y}_i^p\}_{p=k+1-\tilde{M}}^{k}$.
		\State Perform two-loop limited-memory BFGS in Algorithm \ref{alg-loop}.
		\State Output direction  $d_i^{k+1}=H_i^{k+1}g_i^{k+1}$.
	\end{algorithmic}
\end{algorithm}
\begin{algorithm}
	\caption{Two-loop limited-memory BFGS run on agent $i$} \label{alg-loop}
	\begin{algorithmic}
		\State Set $q_i\gets g_i^{k+1}$.
		\For {$p =k,k-1,\ldots,k+1-\tilde{M}$}
		\State $\alpha_i^p \gets \frac{(s_i^p)^Tq_i}{({s}_i^p)^T\hat{y}_i^p}$.
		\State $q_i\gets q_i-\alpha_i^p\hat{y}_i^p$.
		\EndFor
		\State $r_i\gets H_i^{k,(0)}q_i$.
		\For {$p =k+1-\tilde{M},k-\tilde{M},\ldots,k$}
		\State $\beta_i \gets \frac{(\hat{y}_i^p)^T r_i}{({s}_i^p)^T\hat{y}_i^p}$.
		\State $r_i\gets r_i+s_i^p(\alpha_i^p-\beta_i)$.
		\EndFor
		\State Output $H_i^{k+1} g_i^{k+1}=r_i$.
	\end{algorithmic}
\end{algorithm}

\begin{remark}
Compared with the damped regularized limited-memory DFP method, the proposed damp limited-memory BFGS does not use the regularization term parameterized by $\rho$. The reason is that adding such a regularization term $\rho I_d$ at the end of update \eqref{15} will make it difficult to realize the two-loop recursion. How to develop decentralized regularized BFGS which is friendly to two-loop recursion will be our future work.
The  memory requirement and computation cost per iteration of the proposed BFGS are $O(Md)$ and $O(Md)$, respectively. In contrast, the memory requirement and computation cost per iteration of the proposed DFP are $O(d^2+Md)$ and $O(Md^2)$, respectively.
\end{remark}
\subsection{Bounded Positive Eigenvalues of $H_i^k$ Constructed by BFGS}
In the following analysis, we will prove that the Hessian inverse approximations constructed by the proposed damped limited-memory BFGS are positive definite and have bounded eigenvalues.

Consider the update \eqref{15}. The following Lemma shows the damping technique  guarantees that  $({s}_i^p)^T\hat{y}_i^p>0$, which yields $\lambda_{\min}(H_i^{k})>0$. The proof is similar to that of Lemma \ref{lem-dfp-pd}, we write it down here for completeness.

\begin{lemma}\label{lem-bfgs-pd}
	Considering the update in \eqref{15}, with the corrected $\hat{y}_i^p$ by the damping technique, we have
	\begin{equation*}
	\text{$0<\theta_i^p\le 1$ and
	$(s_i^p)^T\hat{y}_i^p\ge 0.25 (s_i^p)^T(H_i^{k,(0)}+\epsilon I_d)^{-1}s_i^p$.}
	\end{equation*}	
	Moreover, with the initialization $H_i^{k,(0)}$ defined in \eqref{w21},  $H_i^{k+1}$ generated by the damped  limited-memory BFGS in Algorithm \ref{alg-BFGS} on each node keeps positive definite and $\lambda_{\min}(H_i^{k+1})>0$.
\end{lemma}
\begin{proof}
	Since the analysis holds for any node, we omit the node index $i$ in the proof.
	With $\tilde{\theta}^p$ defined in \eqref{13}, if $(s^p)^{T}y^p > 0.25 (s^p)^{T}\left(H^{k,(0)}+\epsilon I_d\right)^{-1} s^p$, we have $\tilde{\theta}^p=1$. If $(s^p)^{T}y^p \le 0.25 (s^p)^{T}\left(H^{k,(0)}+\epsilon I_d\right)^{-1} s^p$, substituting this inequality into the definition of $\tilde{\theta}^p$, we have
	\begin{align*}
	\tilde{\theta}^p=&\frac{0.75 (s^p)^{T}\left(H^{k,(0)}+\epsilon I_d\right)^{-1} s^p}{(s^p)^{T}\left(H^{k,(0)}+\epsilon I_d\right)^{-1} s^p-(s^p)^{T}y^p}\\
	\le& \frac{0.75 (s^p)^{T}\left(H^{k,(0)}+\epsilon I_d\right)^{-1} s^p}{(s^p)^{T}\left(H^{k,(0)}+\epsilon I_d\right)^{-1} s^p-[0.25 (s^p)^{T}\left(H^{k,(0)}+\epsilon I_d\right)^{-1} s^p]}\\
	=&1.
	\end{align*}
	Obviously, with $H^{k,(0)}\succ 0$, we have $\tilde{\theta}^p> 0$. Since $
	\theta^p=\min\left\{\tilde{\theta}^p,\frac{\tilde{L}\|s^p\|}{\|y^p\|}\right\}$, with $0<\tilde{\theta}^p\le 1$, we have $0<\theta^p\le 1$.
	
	Moreover, substituting the definitions of $\hat{y}^p$ in \eqref{12} and $\theta^p$ in \eqref{ineq:thetak1}, we compute
	\begin{align}
	&(s^p)^T\hat{y}^p\\\nonumber
	=&(s^p)^T\left(\theta^p y^p+(1-\theta^p)(H^{k,(0)}+\epsilon I)^{-1} s^p\right)\\\nonumber
	=& \theta^p\left[(s^p)^Ty^p-(s^p)^T(H^{k,(0)}+\epsilon I)^{-1}s^p\right]\\\nonumber
	&+(s^p)^T(H^{k,(0)}+\epsilon I)^{-1}s^p\\\nonumber
	=&\left\{\begin{array}{l}
	{0.25 (s^p)^{T}\left(H^{k,(0)}+\epsilon I_d\right)^{-1} s^p}, \\
	\hspace{1em} \text { if } (s^p)^{T}y^p \leq 0.25 (s^p)^{T}\left(H^{k,(0)}+\epsilon I_d\right)^{-1} s^p,\\
	(s^p)^T y^p, \text { otherwise,}
	\end{array}\right.
	\end{align}
	which implies
	\begin{equation*}
	(s^p)^T\hat{y}^p\ge 0.25 (s^p)^{T}\left(H^{k,(0)}+\epsilon I_d\right)^{-1} s^p.
	\end{equation*}
	By the initialization $H^{k,(0)}$ in \eqref{w21}, we have $(s^p)^T\hat{y}^p>0$.
	To show that $H^{k+1}$ is positive definite, for the second inequality \eqref{15} and any nonzero $z\in \mathbb{R}^d$, we have
	\begin{align*}
	&z^TH^{k,(t+1)}z\\\nonumber
	=&z^T\left(I_d-\frac{s \hat{y}^T}{s^T \hat{y}}\right)H\left(I_d-\frac{ \hat{y} s^T}{s^T \hat{y}}\right)z +\frac{(s^Tz)^2}{s^T \hat{y}},
	\end{align*}
	where  we omit all the time indexes at the right-hand side of \eqref{15} for simplicity. With $s^T\hat{y}>0$, we know that $H^{k,(t+1)}$ is positive definite and $\lambda_{\min}(H^{k,(t+1)})>0$, which completes the proof.
\end{proof}

Based on Lemma \ref{lem-bfgs-pd}, the following theorem further gives the specific lower and upper bounds for the eigenvalues of $H_i^k$ generated by Algorithm \ref{alg-BFGS}.
\begin{theorem}\label{thm-bfgs}
Consider the damped limited-memory BFGS method in Algorithm \ref{alg-BFGS}. We have
\begin{equation*}
M_1 I_d\preceq H_i^k\preceq M_2 I_d,
\end{equation*}
	where $M_1= \left(\frac{1}{\beta}+\frac{M\omega^2}{4(\mathcal{B}+\epsilon)}\right)^{-1}$, $M_2=(1+\omega)^{2M}$ $\left(\mathcal{B}+\frac{1}{\tilde{L}(\omega+2)}\right)$ and  $\omega\triangleq4(\mathcal{B}+\epsilon)\left(\tilde{L}+\frac{1}{\beta+\epsilon}\right)$.
\end{theorem}
\begin{proof}
Since the analysis holds for any node, we omit the node index $i$ in the proof.
Note that the inequalities  \eqref{ineq:syp}, \eqref{ineq:hatyp}, \eqref{ineq:ysp-norm} and \eqref{ineq:y2} for the proposed DFP method also hold for the proposed BFGS method, as long as we replace $\hat{s}$ with $s$. Thus,  we directly use these inequalities by replacing $\hat{s}$ with $s$  and omit the proof.

First, we establish the upper bound.
According to the update \eqref{15}, we have
\begin{align}\label{w34}\nonumber
	\|H^{k,(t+1)}\|_2&\leq \|H^{k,(t)}\|_2\cdot\left\|I-\frac{s^p (\hat{y}^p)^T}{(s^p)^T \hat{y}^p}\right\|_2^2+\left\|\frac{{s}^p({s}^p)^T}{({s}^p)^T\hat{y}^p}\right\|_2\\
	&\leq (1+\omega)^2\|H^{k,(t)}\|_2+4(\mathcal{B}+\epsilon),
\end{align}
where we use \eqref{ineq:ysp-norm} and \eqref{ineq:syp} in the last inequality. Following the standard argument for recurrence on \eqref{w34}, we have
\begin{align}\label{w35}\nonumber
	\|H_i^{k,(\tilde{M})}\|_2\leq&(1+\omega)^{2M}\left(\|H_i^{k,(0)}\|_2+\frac{4(\mathcal{B}+\epsilon)}{(1+\omega)^2-1}\right)\\
	\leq&(1+\omega)^{2M}\left(\mathcal{B}+\frac{1}{\tilde{L}(\omega+2)}\right),
\end{align}
where where to derive the last inequality we use $$\frac{4(\mathcal{B}+\epsilon)}{(1+\omega)^2-1}=\frac{1}{(\tilde{L}+\frac{1}{\beta+\epsilon})(\omega+2)}<\frac{1}{\tilde{L}(\omega+2)}.$$ Thus, we obtain the upper bound given by $M_2=(1+\omega)^{2M}$ $\left(\mathcal{B}+\frac{1}{\tilde{L}(\omega+2)}\right)$.

Next, we establish the lower bound. Using the Sherman-Morrison-Woodbury formula on \eqref{15}, we get
\begin{align}\label{eq:bfgs-inv}	
	&\left(H^{k,(t+1)}\right)^{-1}\\\nonumber
	=&\left(H^{k,(t)}\right)^{-1}+\frac{\hat{y}^p(\hat{y}^p)^T}{(s^p)^T\hat{y}^p}-\frac{\left(H^{k,(t)}\right)^{-1}s^p(s^p)^T\left(H^{k,(t)}\right)^{-1}}{(s^p)^TH^{k,(t)}s^p},
\end{align}
which implies
\begin{align}\label{w37}
	\left\|\left(H^{k,(t+1)}\right)^{-1}\right\|_2\leq&\left\|\left(H^{k,t}\right)^{-1}\right\|_2+\left\|\frac{\hat{y}^p(\hat{y}^p)^T}{(s^p)^T\hat{y}^p}\right\|_2\\\nonumber
	\leq &\left\|\left(H^{k,t}\right)^{-1}\right\|_2+\frac{\omega^2}{4(\mathcal{B}+\epsilon)},
\end{align}
where we use \eqref{ineq:y2} in the last inequality. Following the standard argument for recurrence on \eqref{w37}, we have
\begin{equation*}
\left\|\left(H_i^{k,(\tilde{M})}\right)^{-1}\right\|_2\leq\left\|\left(H_i^{k,0}\right)^{-1}\right\|_2+\frac{M\omega^2}{4(\mathcal{B}+\epsilon)},
\end{equation*}
which implies
\begin{equation*}
\lambda_{\min}\left(H_i^{k,(\tilde{M})}\right)\geq\left(\frac{1}{\beta}+\frac{M\omega^2}{4(\mathcal{B}+\epsilon)}\right)^{-1}.
\end{equation*}
Therefore, we obtain the lower bound given by $M_1=\left(\frac{1}{\beta}+\frac{M\omega^2}{4(\mathcal{B}+\epsilon)}\right)^{-1}$ and complete the proof.
\end{proof}

\begin{remark}\label{remark:2}
Regarding the theoretical results given by Theorems \ref{thm-dfp} and \ref{thm-bfgs}, we have the following comments.
\begin{enumerate}
\item
 The Hessian inverse approximations $H_i^k$ constructed by the proposed DFP and BFGS methods satisfy Assumption \ref{asm-Hbound} and thus fit into the general framework in Part I for exact linear convergence.
\item  With a large memory size $M$, the eigenvalues of $H_i^k$ have a wide range and thus $H_i^k$ may be almost singular. {One conjecture is that noise caused by randomness and disagreement accumulates more with larger $M$.} On the other hand, observe from the updates of the proposed quasi-Newton methods that a too small memory size $M$ may lead to insufficient second-order curvature information. Thus, we recommend to use a moderate memory size $M$, which leads to low computation and memory costs without sacrificing the performance.
\item For the proposed DFP method, the regularization term $\rho I_d$ in \eqref{a13} lifts the lower bound $M_1$ by $\rho$ and lifts the upper bound $M_2$ by $M\rho$. It is the limited-memory technique that prevents the regularization term $\rho I_d$ from accumulating to infinity.  The analysis of Theorem \ref{thm-dfp} also holds for $\rho=0$. However, we observe from the numerical experiments that a suitably tuned $\rho>0$ can improve the performance.
\item The analysis in Theorems \ref{thm-dfp} and \ref{thm-bfgs} also hold for $\epsilon=0$. However, we observe from the numerical experiments that a suitably tuned $\epsilon>0$ can improve the performance, especially for the proposed BFGS method.
\end{enumerate}
\end{remark}
\section{numerical experiments}
In this section, we embed the proposed DFP method in Algorithm \ref{alg-DFP} and the proposed BFGS method in Algorithm \ref{alg-BFGS} into the general framework of Part I and evaluate their performance. We use the two resultant algorithms to solve a least-squares problem with synthetic data in Section \ref{num-condition}  and a logistic regression problem with real data in Section \ref{num-first-order}--\ref{num-topo}.  We randomly generate a connected and undirected network with $n$ nodes and $\frac{\varrho n(n-1)}{2}$ edges, where $\varrho\in(0,1]$ is the connectivity ratio.
The performance metric is the relative error defined as
$$\text{relative error}=\frac{\left\|\mathbf{x}^{k}-\mathbf{x}^{*}\right\|^2} {n\left\|\mathbf{x}^{0}-\mathbf{x}^{*}\right\|^2},$$
where the optimal solution $x^*$ is pre-computed through a centralized Newton method.
\subsection{Effects of Condition Number}\label{num-condition}
We consider a least-squares problem defined as
\begin{equation*}
 \underset{{x} \in \mathbb{R}^{d}}{\operatorname{min}}  \frac{1}{2}\sum_{i=1}^{n}\|A_ix-b_i\|^{2},
\end{equation*}
where $A_i\in \mathbb{R}^{m\times d}$ and $b_i\in \mathbb{R}^m$ are synthetic data privately owned by node $i$. Here, we set $m=500$ and $d=8$. For simplicity, we define aggregated variables   $A=[A_1;\cdots;A_n]\in \mathbb{R}^{nm\times d}$ and $b=[b_1;\cdots;b_n]\in \mathbb{R}^{nm}$ by vertically  stacking the local data. We define the condition number of the problem as $$\kappa_{LS}=\frac{\lambda_{\max}(A^TA)}{\lambda_{\min}(A^TA)}.$$
To show the effects of the condition number, we generate two groups of data with $\kappa_{LS}=10$ and $\kappa_{LS}=2000$ as follows. For $\kappa_{LS}=10$, we fix $\lambda_{\min}(A^TA)=0.1$ and  $\lambda_{\max}(A^TA)=1$. For $\kappa_{LS}=2000$, we fix $\lambda_{\min}(A^TA)=0.001$ and  $\lambda_{\max}(A^TA)=2$.  The other $(d-2)$ eigenvalues are randomly generated within the interval $\big[\lambda_{\min}(A^TA),\lambda_{\max}(A^TA)\big]$. Figs. \ref{kappa_10} and \ref{kappa_2000} record the results for $\kappa_{LS}=10$ and $\kappa_{LS}=2000$, respectively.  The parameters are set as follows. We set $n=20$ and the connectivity ratio is $\varrho=0.5$. We set $\mathcal{B}=10^4$ for the two proposed quasi-Newton methods. When $\kappa_{LS}=10$ $(2000)$,
for the proposed DFP method, we set $\alpha=0.6 $ $(0.6)$, $\rho=10^{-5}$ $ (10^{-5})$, $\epsilon=3 $ $ (5)$, $\beta=0.04$ $ (0.01)$,
$\tilde{L}=10$ $(10)$,
the memory size $M=20$ $ (20)$, the batch size  $b_i=10$ $ (15)$.
For the proposed BFGS method, we set $\alpha=0.6 $ $(0.6)$, $\epsilon=3 $ $ (37)$, $\beta=0.04$ $ (0.01)$, $\tilde{L}=10$ $(10)$, the memory size $M=20$ $ (50)$, the batch size $b_i=10$ $ (15)$, respectively. We compare with the existing first-order methods, DSA \cite{mokhtari2016dsa}, GT-SVRG and GT-SAGA \cite{xin2020variance},  and  Acc-VR-DIGing \cite{li2020optimal}.
For DSA, we set the stepsize $\alpha=0.9 $ $(0.9)$ and the batch size $b_i=1 $ $(1)$. For GT-SVRG, we set the stepsize  $\alpha=0.9$ $ (0.9)$ and the batch size $b_i=1$ $ (1)$. For GT-SAGA, we set the stepsize $\alpha=0.95 $ $(0.95)$ and the batch size $b_i=1 $ $(1)$. For Acc-VR-DIGing, we set the step size $\alpha=0.9 $ $(0.9)$, the two parameters $\theta_1=0.2 $ $(0.1)$ and $\theta_2=0.01 $ $(0.01)$, while the batch size $b_i=1 $ $(1)$. Note that all the first-order methods use batch size $b_i=1$, which yields the faster convergence in terms of the number of epochs in this set of numerical experiments.

From Figs. \ref{kappa_10} and \ref{kappa_2000},  the proposed quasi-Newton methods outperform the existing first-order methods, and their advantages are more obvious for the ill-conditioned problem. The proposed DFP method performs better than the BFGS method in terms of the number of epochs, but the BFGS method has lower computation and storage complexity, as we have discussed in Remark \ref{remark:2}.

\begin{figure}
	\centering
	\centerline{\includegraphics[width=6cm]{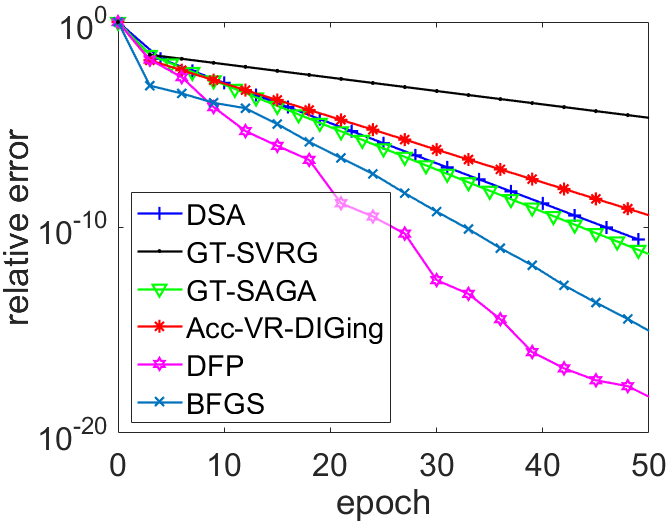}}
	\caption{Least-squares problem with $\kappa_{LS}=10$.}
	\label{kappa_10}
\end{figure}
\begin{figure}
	\centering
	\centerline{\includegraphics[width=6cm]{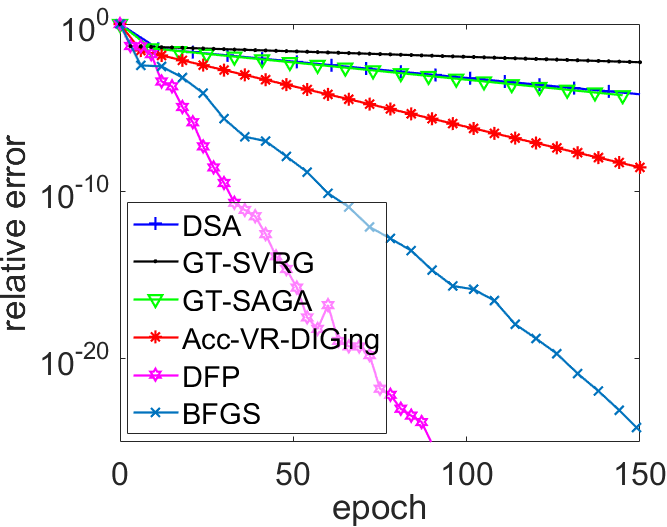}}
	\caption{Least-squares problem with $\kappa_{LS}=2000$.}
	\label{kappa_2000}
\end{figure}

\subsection{Comparison with First-order Algorithms: Real Datasets}\label{num-first-order}
The ensuing numerical experiments evaluate on the real datasets.  We use the proposed quasi-Newton methods to solve a logistic regression problem in the form of
\begin{align*}\label{logreg}
 \underset{{x} \in \mathbb{R}^{d}}{\operatorname{min}}  \frac{\iota}{2}\|{x}\|^{2}+\frac{1}{n}\!\sum_{i=1}^{n} \!\frac{1}{m_i}\! \sum_{j=1}^{m_{i}}  \ln \left(1+\exp \left(\!-\!\left(\mathbf{o}_{i j}^{T} {x}\right) \mathbf{p}_{i j}\right)\right),
\end{align*}
where node $i$ privately owns $m_i$ training samples $(\mathbf{o}_{i l}, \mathbf{p}_{i l}) \in \mathbb{R}^{d} \times\{-1,+1\}$; $l=1, \ldots, m_{i}$.
We use five real datasets\footnote{https://www.csie.ntu.edu.tw/\~{}cjlin/libsvmtools/datasets/}, whose attributes are summarized in Table \ref{table}. We normalize each sample such that $\|\mathbf{o}_{i l}\|=1, \forall i,l$. Note that another way is to normalize each feature, which yields better condition number but is nontrivial to implement in the decentralized setting.  A regularization term $\frac{\iota}{2}\|{x}\|^{2}$ with $\iota>0$ is used to avoid over-fitting.
We set $n=20$, $\iota=0.001$ and $\mathcal{B}=10^4$ throughout the following numerical experiments. The training samples are randomly and evenly distributed over all the $n$ nodes.
\begin{table}
	\centering
	\caption{Datasets used in numerical experiments.}\label{table}
	\begin{tabular}{c|c|c}
		\hline
		\bf Dataset  & \bf \# of Samples 	
		($\sum_{i=1}^{n} m_i$) & \bf \# of Features ($d$)\\\hline 		
		covtype &40000  & 54 \\
		cod-rna &52000  & 8 \\
		a6a & 11220 & 123\\
		a9a & 32560 & 123\\
		ijcnn1 & 91700 & 22\\
		\hline
	\end{tabular}
\end{table}

We compare the proposed two decentralized stochastic quasi-Newton methods with four decentralized stochastic first-order algorithms, DSA in \cite{mokhtari2016dsa}, GT-SVRG and GT-SAGA in  \cite{xin2020variance},  and  Acc-VR-DIGing in \cite{li2020optimal} on four real datasets. Different to the numerical experiments on the synthetic data where the batch sizes are set as 1, now we use larger batch sizes to boost the performance of the four first-order methods.
Figs. \ref{covtype}--\ref{a9a} depict the results on covtype, cod-rna, a6a and a9a, respectively.  The parameters are set as follows.  We set $n=20$ and the connectivity ratio $\varrho=0.5$.  When the dataset is covtype (cod-rna, a6a, a9a),
for the proposed DFP method, we set $\alpha=0.32 $ $(0.3, 0.38,0.38)$, $\rho=0.01$ $ (0.0002, 0.01, 0.001)$, $\epsilon=0.02 $ $ (0.03,0.005, 0.1)$, $\beta=0.002$ $ (0.002, 0.015, 0.5)$,
$\tilde{L}=50$ $(50, 20, 50)$,
the memory size $M=3$ $ (20, 40,  50)$, the batch size ratio $b_i/m_i=10\%$ $ (8\%, 10\%, 6\%)$, respectively.
For the proposed BFGS method, we set $\alpha=0.37 $ $(0.35, 0.38, 0.35)$, $\epsilon=0.001 $ $ (100, 30, 30)$, $\beta=0.002$ $ (0.002, 1.2, 0.5)$, $\tilde{L}=50$ $(50, 20, 20)$, the memory size $M=3$ $ (40, 50, 50)$, the batch size ratio $b_i/m_i=10\%$ $ (10\%, 10\%,  10\%)$, respectively.
For GT-SVRG, we set the stepsize  $\alpha=0.002$ $ (0.01, 0.009, 0.004)$ and the batch size $b_i=5$ $ (2, 1, 2)$. For DSA, we set the stepsize $\alpha=0.001 $ $(0.03, 0.009, 0.008)$ and the batch size $b_i=10 $ $(10, 1, 2)$. For GT-SAGA, we set the stepsize $\alpha=0.002 $ $(0.009, 0.009, 0.0035)$ and the batch size $b_i=5 $ $(2, 1, 1)$. For Acc-VR-DIGing, we set the stepsize $\alpha=0.002 $ $(0.03, 0.04, 0.015)$, the two parameters $\theta_1=0.9 $ $(0.07, 0.1, 0.1)$ and $\theta_2=0.01 $ $(0.05, 0.1, 0.1)$, while the batchsize $b_i=5 $ $(10, 5, 5)$.

As Figs. \ref{covtype}--\ref{a9a} show,  the proposed two decentralized stochastic quasi-Newton methods outperform  DSA, GT-SVRG, GT-SAGA and Acc-VR-DIGing in all the four datasets, demonstrating the gains of curvature information from the constructed Hessian inverse approximations.  Generally, the proposed DFP method is better than the proposed BFGS method in all the four datasets, but the BFGS method has lower memory requirement and lower computation cost.
\begin{figure}
	\centering
	\centerline{\includegraphics[width=6cm]{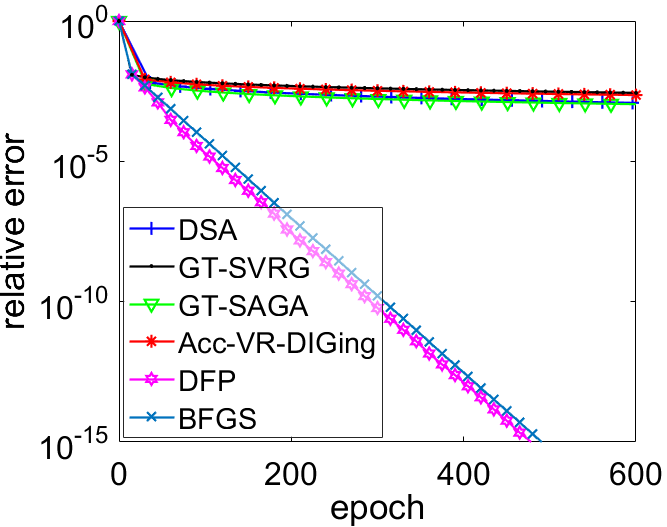}}
	\caption{Comparison with first-order algorithms on covtype.}
	\label{covtype}
\end{figure}
\begin{figure}
	\centering
	\centerline{\includegraphics[width=6cm]{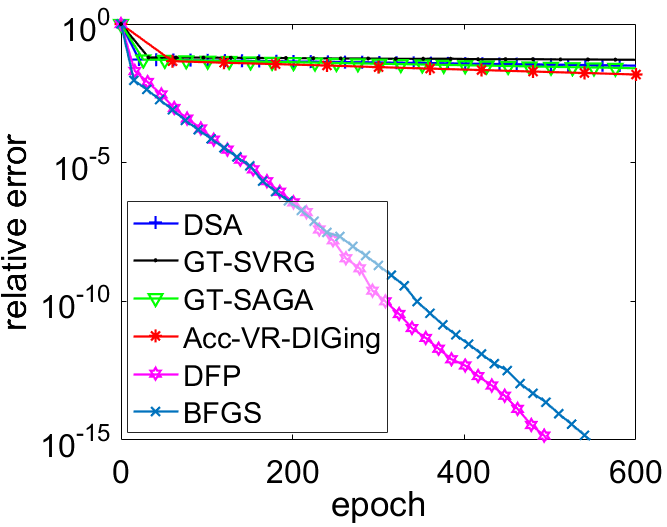}}
	\caption{Comparison with first-order algorithms on cod-rna.}
	\label{cod-rna}
\end{figure}
\begin{figure}
	\centering
	\centerline{\includegraphics[width=6cm]{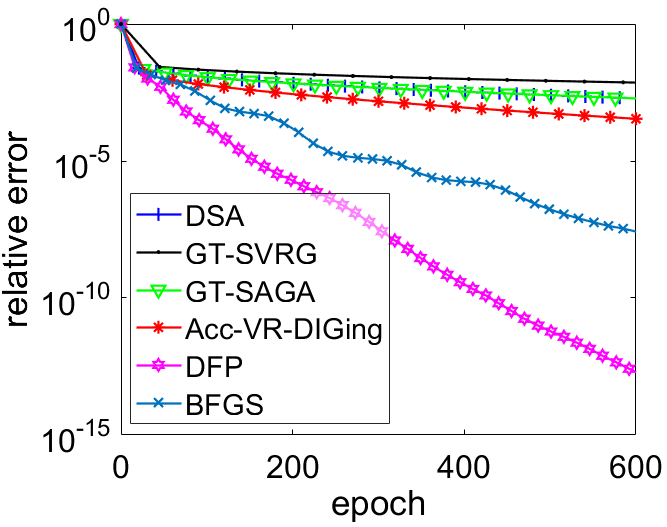}}
	\caption{Comparison with first-order algorithms on a6a.}
	\label{a6a}
\end{figure}
\begin{figure}
	\centering
	\centerline{\includegraphics[width=6cm]{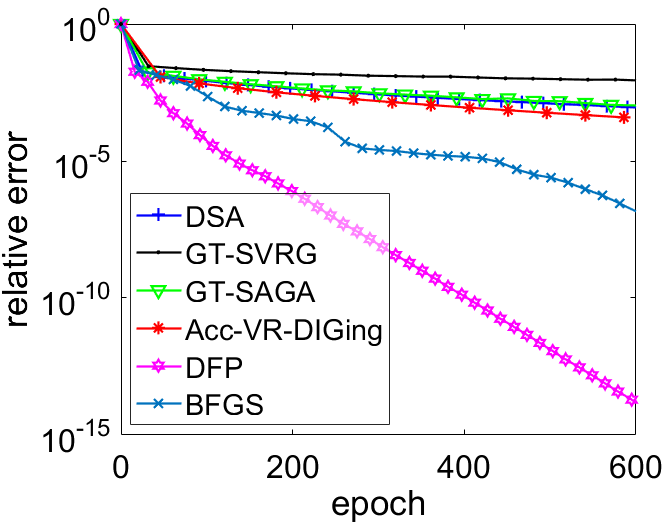}}
	\caption{Comparison with first-order algorithms on a9a.}
	\label{a9a}
\end{figure}
\subsection{Effects of Batch Size}\label{num-batch}
%
Here, we numerically show the effects of batch size on the performance of proposed DFP and BFGS methods for solving the logistic regression problem using the real dataset ijcnn1.
In Figs. \ref{batch-DFP} and \ref{batch-BFGS}, we evaluate the effects of different batch size ratios ( $b_i/m_i=2\%, 4\%, 6\%, 8\%$ and $10\%$), on the performance of the proposed DFP and BFGS methods, respectively. The parameters are set as follows. For the proposed DFP (BFGS) method, we set $\alpha=0.32 $ $(0.31)$, $\rho=0.005$ $ (0)$, $\epsilon=0.005 $ $ (0.005)$, $\beta=0.1$ $ (0.1)$, and the memory size $M=50$ $ (50)$. The other settings are the same as those used in Fig. \ref{covtype}.

From Figs. \ref{batch-DFP} and \ref{batch-BFGS}, we observe that too larger or smaller batch sizes lead to more epochs, because a smaller batch size causes higher stochastic gradient noise, while a larger batch size calls for more gradient evaluations per iteration. For both of the proposed DFP and BFGS methods, a batch size ratio of $b_i/m_i = 6\%$ gives the best performance.
\begin{figure}
	\centering
	\centerline{\includegraphics[width=6cm]{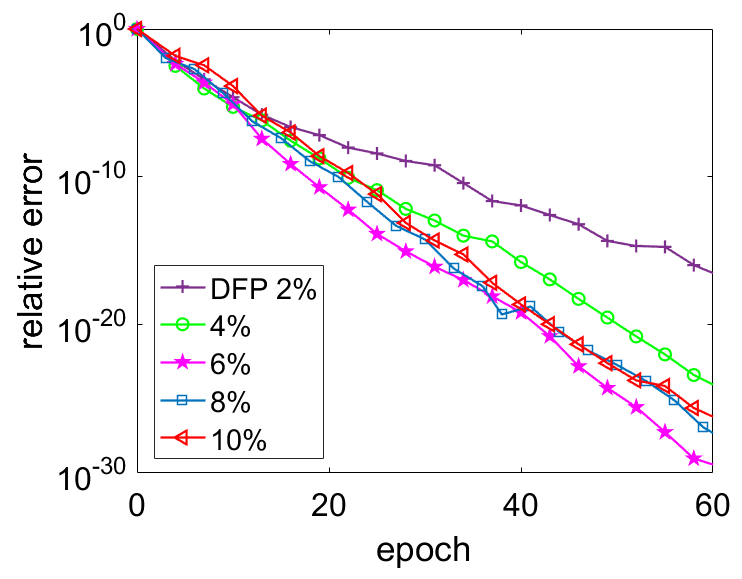}}
	\caption{Effects of batch size of DFP on ijcnn1.}
	\label{batch-DFP}
\end{figure}
\begin{figure}
	\centering
	\centerline{\includegraphics[width=6cm]{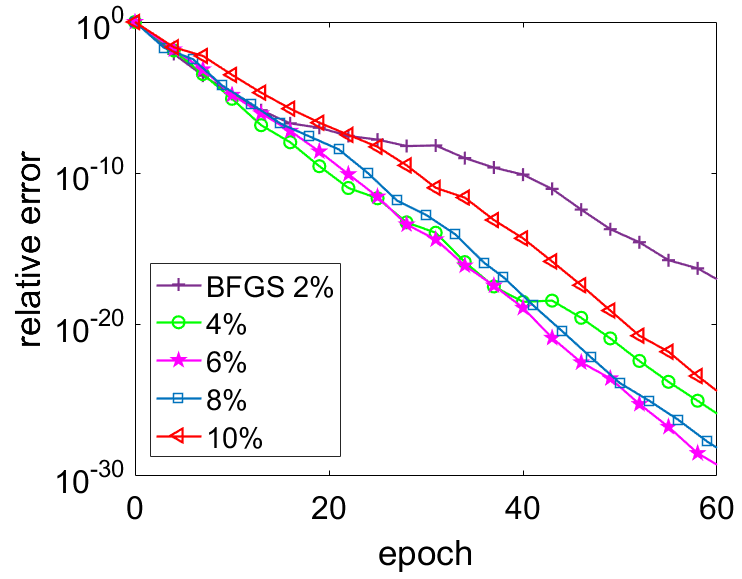}}
	\caption{Effects of batch size of BFGS on ijcnn1.}
	\label{batch-BFGS}
\end{figure}

\subsection{Effects of Memory Size}\label{num-memory}
In Figs. \ref{memory-DFP} and \ref{memory-BFGS}, we evaluate the effects of different memory sizes ( $M=5, 10, 20, 30, 40,$ and $50$) on the performance of the proposed DFP and BFGS methods, respectively. The problem is logistic regression and the dataset is ijcnn1. The parameters are set as follows. For the proposed DFP (BFGS) method, we set $\alpha=0.32 $ $(0.31)$, $\rho=0.004$ $ (0)$, $\epsilon=0.005 $ $ (0.002)$, $\beta=0.001$ $ (0.1)$, and the batch size ratio $b_i/m_i=6\%$ $ (6\%)$. The other settings are the same as those used in Fig. \ref{covtype}.

As  Figs. \ref{memory-DFP} and \ref{memory-BFGS} show, a larger memory size generally leads to faster convergence, but the improvement becomes marginal when the memory size is sufficiently large. Therefore, we can use a moderate memory size, which leads to low memory and computation costs.
\begin{figure}
	\centering
	\centerline{\includegraphics[width=6cm]{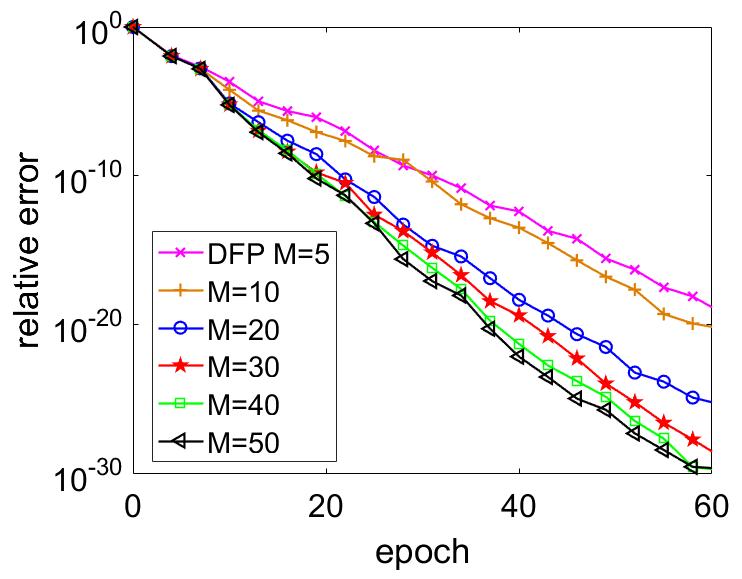}}
	\caption{Effects of memory size of DFP on ijcnn1.}
	\label{memory-DFP}
\end{figure}
\begin{figure}
	\centering
	\centerline{\includegraphics[width=6cm]{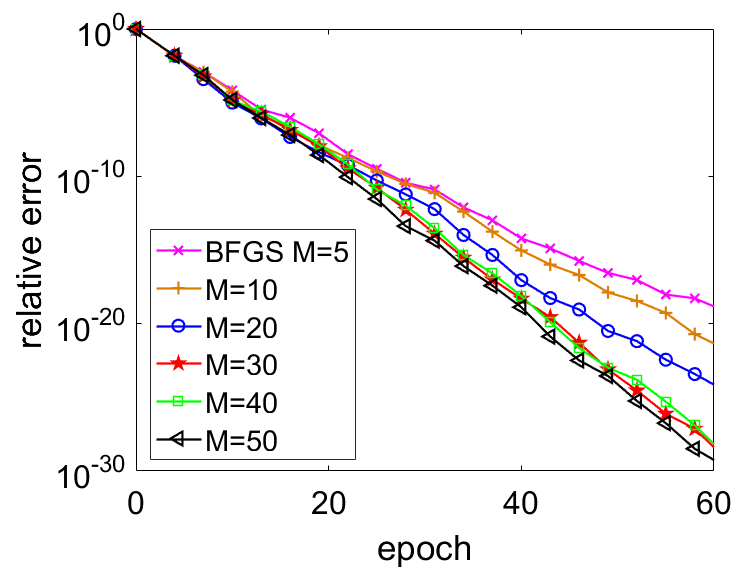}}
	\caption{Effects of memory size of BFGS on ijcnn1.}
	\label{memory-BFGS}
\end{figure}
\subsection{Effects of Topology}\label{num-topo}
In Figs. \ref{topo-DFP} and \ref{topo-BFGS}, we evaluate the effects of five different topologies (cycle,star, random graphs with connectivity ratios $\varrho=0.2$, $0.3$, $0.5$) on the performance of the proposed DFP and BFGS methods, respectively. The second largest singular values $\sigma$ of $W$ i.e.,
$\sigma=\|W-\frac{1}{n}1_n 1_n^T\|_2$, of the five graphs are $\sigma=0.967$, $0.950$, $0.863$, $0.797$, and $0.569$, respectively.
The parameters are set as follows. For the proposed DFP method, we set $\alpha=0.035$ $(0.02, 0.2, 0.25, 0.32)$, $\rho=0.003$ $(0.001, 0.001, 0.001, 0.005)$, $\epsilon=0.005$ $(0.005, 0.005, 0.005, 0.005)$, $\beta=0.1$ $(0.1, 0.1, 0.1, 0.1)$, the memory size $M=50$ $(50, 50, 50, 50)$,  and the batch size ratio $b_i/m_i=6\%$ $(6\%,6\%,6\%,6\%)$.
For the proposed BFGS method, we set $\alpha=0.06$ $(0.07, 0.2, 0.3, 0.31)$, $\epsilon=0.005$ $(0.005, 0.002, 0.002, 0.002)$, $\beta=0.1$ $(0.1, 0.1, 0.1, 0.1)$, $M=50$ $(50, 50, 50, 50)$, and the batch size ratio $b_i/m_i=11\%$ $(10\%,6\%,6\%,6\%)$. The other settings are the same as those used in Fig. \ref{covtype}.

From Figs. \ref{topo-DFP} and \ref{topo-BFGS}, we observe that the proposed two decentralized quasi-Newton methods converge linearly on different graphs. For both methods, graphs with smaller $\sigma$ give faster convergence rates, which corroborate with the theoretical results in Part I.
\begin{figure}
	\centering
	\centerline{\includegraphics[width=6cm]{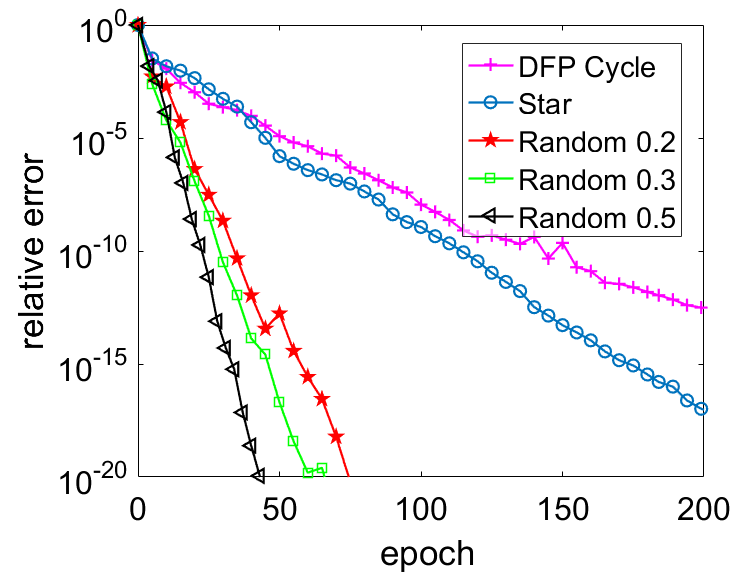}}
	\caption{Effects of topology of DFP on ijcnn1.}
	\label{topo-DFP}
\end{figure}
\begin{figure}
	\centering
	\centerline{\includegraphics[width=6cm]{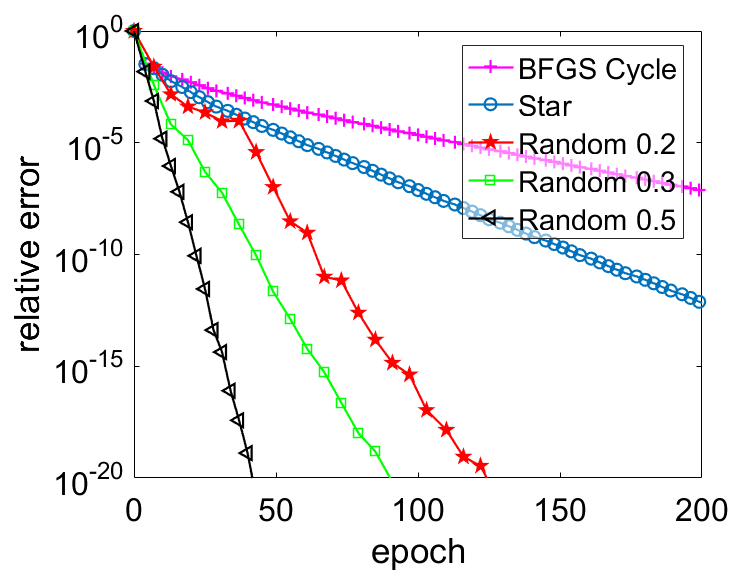}}
	\caption{Effects of topology of BFGS on ijcnn1.}
	\label{topo-BFGS}
\end{figure}
\section{Conclusions}
In Part II of this work, we propose two fully decentralized quasi-Newton methods, damp regularized limited-memory DFP  and damp limited-memory BFGS, to locally construct the Hessian inverse approximations. We use the damping and limited-memory techniques to ensure that the constructed Hessian inverse approximations are positive definite with bounded eigenvalues. For the DFP-based method, we add a regularization term to improve the performance. For the BFGS-based method, we use a two-loop recursion to reduce the memory and computation costs. We prove that quasi-Newton approximations satisfy the assumption in Part I for the exact linear convergence. Numerical experiments in Part II demonstrate that the proposed quasi-Newton methods are much faster than the existing decentralized stochastic first-order methods.


\bibliographystyle{IEEEtran}
\bibliography{IEEEabrv_PnADMM}



%








\end{document}